\documentclass{article}
\usepackage{graphicx} 
\usepackage{amsmath,amssymb,amsfonts,amsthm}
\usepackage{mathrsfs}
\usepackage{pifont}
\usepackage{geometry}

\theoremstyle{definition} \newtheorem{defi}{Definition}[section] 

\newtheorem{exam}{Example}[section]
\theoremstyle{plain} \newtheorem{prop}{Proposition}[section]
\newtheorem{thm}[prop]{Theorem}
\newtheorem{lem}[prop]{Lemma}

\newcommand{\ble}{\begin {lem}}
\newcommand{\ele}{\end {lem}}

\newcommand{\ptdx}[2]{\frac{\partial #1}{\partial x_{#2}}}
\newcommand{\aut}{\operatorname{Aut}}
\newcommand{\ch}{\operatorname{Char}}
\newcommand{\be}{\begin {equation}}
\newcommand{\ee}{\end {equation}}
\newcommand{\bp}{\begin {proof}}
\newcommand{\ep}{\end {proof}}

\title{A symmetric biderivation structure on polynomial algebras and a class of modules over the special Jordan algebra $H_n(K)$ of symmetric matrices} 
\author{Yangjie Yin\\School of Mathematics \\ Zhejiang
	University\\ yj\_yin@zju.edu.cn
    \and Gang Han\thanks{Corresponding Author}\\School of Mathematics \\ Zhejiang
    University\\mathhgg@zju.edu.cn }
\date{September 27, 2025 }

\begin{document}

\maketitle

\begin{abstract}
There exists a biderivation structure on the polynomial algebra
$
\mathscr{A}[n] = K[x_1,\dots,x_n],
$
where $K$ is a field with $\operatorname{char}(K)\ne 2$, defined by
$
f \circ h = \sum_{i=1}^n \frac{\partial f}{\partial x_i}\,\frac{\partial h}{\partial x_i}.
$
Let $\mathscr{A}_k[n]$ denote the subspace of homogeneous polynomials of degree $k$. Then
 $(\mathscr{A}_2[n],\circ)$ is a Jordan algebra, isomorphic to the special Jordan algebra $H_n(K)$ of $n\times n$ symmetric matrices. Each $\mathscr{A}_k[n]$ is a natural $\mathscr{A}_2[n]$-bimodule, which admits a weight space decomposition with respect to a complete set of mutually orthogonal idempotents. In particular, the weight space decomposition of $\mathscr{A}_2[n]$ coincides with its Peirce decomposition. $\mathscr{A}_k[n]$ is a Jordan bimodule if and only if $k=0,1,2$. Equivalently, for all $k\ge 3$, $\mathscr{A}_k[n]$ is not a Jordan bimodule.
 The group of algebra automorphisms of $(\mathscr{A}[n],\cdot,\circ)$ that preserve each homogeneous component $\mathscr{A}_k[n]$ is isomorphic to the orthogonal group $O(n,K)$.
 If $\operatorname{char}(K)=0$, then the algebra $(\mathscr{A}[n],\cdot,\circ)$ is simple, i.e., it has no nonzero proper ideals. Moreover, in this case, each $\mathscr{A}_k[n]$ is a simple $\mathscr{A}_2[n]$-bimodule.
\end{abstract}

\section{Introduction}

Biderivation structures naturally arise in many areas of mathematics, particularly in the study of commutative associative algebras. Classical examples include Poisson algebras, where a commutative associative algebra is equipped with a Poisson bracket, and the Beltrami bracket on the algebra of smooth functions $C^\infty(M)$ over a Riemannian manifold $(M,g)$, introduced by Crouch \cite{crouch1981} (see also \cite{van1984}), defined by
\begin{equation}\label{10}
    [f,h] := g(\nabla f, \nabla h), \qquad f,h \in C^\infty(M),
\end{equation}
where $\nabla f$ denotes the gradient of $f$. In 1986, Morrison \cite{morrison1986} introduced the metriplectic system, a bilinear bracket given by the sum of a Poisson bracket and a symmetric bracket, providing a framework for dynamical systems with dissipation (see also \cite{guha2007,bloch2024}). More generally, \cite{ortega2004} defines a Leibniz bracket, a bilinear biderivation on $C^\infty(M)$ encompassing these cases. These studies focus primarily on the associated dynamical systems, leaving the algebraic properties and related module structures largely unexplored.

In this paper, we investigate the polynomial algebra
$
\mathscr{A}[n] = K[x_1,\dots,x_n]
$
over a field $K$ of characteristic not equal to 2, equipped with the bilinear product
\[
f \circ h := \sum_{i=1}^n \frac{\partial f}{\partial x_i}\frac{\partial h}{\partial x_i}.
\]
This operation, inspired by the Beltrami bracket, defines a commutative but non-associative biderivation. We refer to $(\mathscr{A}[n],\cdot,\circ)$ as the \textit{polynomial algebra with the standard symmetric biderivation structure}. When $\ch(K)=0$, this algebra is simple (Proposition \ref{simp}).

Let $\mathscr{A}_i[n]$ denote the subspace of homogeneous polynomials of degree $i$ in $\mathscr{A}[n]$. Then
\begin{equation}\label{11}
    \mathscr{A}_i[n] \circ \mathscr{A}_j[n] \subset \mathscr{A}_{i+j-2}[n],
\end{equation}
so that $\mathscr{A}_2[n]$ is closed under $\circ$. By computing the Lie brackets of gradient vector fields, $(\mathscr{A}_2[n],\circ)$ is shown to be a Jordan algebra, isomorphic to the special Jordan algebra $H_n(K)$ of symmetric $n \times n$ matrices (Theorem \ref{ThmJA2}). Each $\mathscr{A}_k[n]$ naturally becomes an $\mathscr{A}_2[n]$-bimodule, which is a Jordan bimodule precisely when $k=0,1,2$ (Theorem \ref{ThmJM}), and all bimodules are simple when $\ch(K)=0$ (Theorem \ref{simpMod}). Note that  the  theory of Jordan bimodules over finite dimensional semisimple Jordan algebras was developed by Jacobson in \cite{jacobson1968}. Subsequently, the study of Jordan bimodules over simple Jordan algebras via weights relative to a fixed Cartan subalgebra was carried out in \cite{g1978}. But in our case, the  $\mathscr{A}_2[n]$-bimodule $\mathscr{A}_k[n]$ is not a Jordan bimodule for $k\ge3$.

With respect to the complete set of mutually orthogonal primitive idempotents
$
\left\{\frac{1}{4}x_1^2, \dots, \frac{1}{4}x_n^2\right\} 
$ of $\mathscr{A}_2[n]$,
each $\mathscr{A}_k[n]$ admits a weight space decomposition (Theorem \ref{ThmWSD}). In particular, in the case $k=2$, the decomposition  coincides with the Peirce decomposition for $\mathscr{A}_2[n]$. Finally, the group of automorphisms of $(\mathscr{A}[n],\cdot,\circ)$ preserving each homogeneous component  $\mathscr{A}_k[n]$, $\aut^0(\mathscr{A}[n],\cdot,\circ)$, is shown to be isomorphic to the orthogonal group $O(n,K)$ (Theorem \ref{ThmAut}).

This paper is organized as follows.

In Section \ref{SecAbs}, we recall the notion of algebras with a biderivation structure and provide examples from differential and algebraic geometry. We then examine basic properties of the polynomial algebra with the standard symmetric biderivation structure. In particular, this algebra is neither associative nor Lie. Moreover, when $\ch(K)=0$, the algebra $(\mathscr{A}[n],\cdot,\circ)$ is simple for all $n \ge 1$.

Section \ref{SecJAJM} is devoted to the  structure of $\mathscr{A}_{\le 2}[n]$ under the induced biderivation. By computing the Lie brackets of gradient vector fields, we show that it forms a Jordan algebra and focus on its  subalgebra $\mathscr{A}_2[n]$, which is isomorphic to the special Jordan algebra $H_n(K)$. We also  show that $\mathscr{A}_k[n]$ is a Jordan bimodule over  $\mathscr{A}_2[n]$ if and only if  $k=0,1,2$.

In Section \ref{SecRSD}, we study the weight space decomposition of $\mathscr{A}[n]$ as a bimodule over $\mathscr{A}_2[n]$. By choosing a complete set of mutually orthogonal primitive idempotents in $\mathscr{A}_2[n]$, which form a basis of a Cartan subalgebra, we determine the weights and weight space decompositions of $\mathscr{A}[n]$ and its homogeneous components $\mathscr{A}_k[n]$. In particular, the decomposition of $\mathscr{A}_2[n]$ coincides with its Peirce decomposition. We further show that each bimodule $\mathscr{A}_k[n]$ is simple in the case when $\ch(K)=0$.

Finally, in Section \ref{SecAutGp}, we describe the structure of the subgroup $\aut^0(\mathscr{A}[n])$ of $\aut(\mathscr{A}[n])$ that preserves the homogeneous components, and the full automorphism group $\aut(\mathscr{A}[1])$. We show that $\aut^0(\mathscr{A}[n])$ is isomorphic to the orthogonal group $O(n,K)$.

\section{Algebras with a biderivation structure} \label{SecAbs}

Let us recall the definition of a biderivation structure on  a commutative associative algebra.

\begin{defi}
Let $K$ be a field and $(\mathscr{A},\cdot)$ a commutative associative algebra over $K$. A bilinear map
$\Psi: \mathscr{A} \times \mathscr{A} \to \mathscr{A}$
is called a \textbf{biderivation structure} on $(\mathscr{A},\cdot)$ if, for each $f \in \mathscr{A}$, the maps
\begin{align*}
    \Psi(f,\cdot): \mathscr{A} \to \mathscr{A}, &\quad g \mapsto \Psi(f,g);\\
    \Psi(\cdot,f): \mathscr{A} \to \mathscr{A}, &\quad g \mapsto \Psi(g,f),
\end{align*}
are $K$-derivations of $(\mathscr{A},\cdot)$. In this case, the triple $(\mathscr{A},\cdot,\Psi)$ is called an \textbf{algebra with a biderivation structure}. 

The biderivation $\Psi$ is said to be \textbf{symmetric} (resp. \textbf{skew-symmetric}) if
$\Psi(f,g) = \Psi(g,f)$ (resp. $\Psi(f,g) = -\Psi(g,f)$)
for all $f,g \in \mathscr{A}$.
\end{defi}

The following examples from differential and algebraic geometry illustrate and motivate the definition of a biderivation structure.

\begin{exam}
Let $(P,\omega)$ be a symplectic manifold, and let $C^\infty(P)$ denote the algebra of smooth functions on $P$. For each $f \in C^\infty(P)$, let $X_f$ denote its Hamiltonian vector field, defined by $\omega(X_f,Y) = df(Y)$ for all smooth vector fields $Y$ on $P$. 
Then the algebra $C^\infty(P)$ becomes a Poisson algebra equipped with the Poisson bracket $\{f,g\} := \omega(X_f,X_g)$.
In particular, $\{\cdot,\cdot\}$ defines a skew-symmetric biderivation structure on $C^\infty(P)$.
\end{exam}

\begin{exam}\label{Rmm}
Let $(M,g)$ be a (pseudo-)Riemannian manifold, and let $C^\infty(M)$ denote the algebra of smooth functions on $M$. For $f \in C^\infty(M)$, let $\nabla f$ denote the gradient vector field of $f$, i.e., $g(\nabla f,Y) = df(Y)$ for all smooth vector fields $Y$ on $M$.
Define a bilinear operation
\begin{align*}
        [\cdot,\cdot]\colon \quad C^\infty(M) \times C^\infty(M) \quad &\longrightarrow \quad C^\infty(M),\\
        (f,h) \quad \quad \quad \quad &\longmapsto \quad g(\nabla f , \nabla h).
\end{align*}
Then $[\cdot,\cdot]$ is a symmetric biderivation structure on $C^\infty(M)$, commonly referred to as the Beltrami bracket \cite{crouch1981}.
\end{exam}

Now we give a generalization of above examples.
\begin{exam}\label{3}
Let $(X,h)$ be a smooth affine variety over $K$ with a nowhere-degenerate symmetric tensor field $h$ ($\in\Gamma(X,\operatorname{Sym}^2\Omega^1_X)$) of type $(0,2)$ on $X$, and $K[X]$ be the coordinate ring over $X$. The following operation
  \begin{align*}
	 \quad K[X] \times K[X] \quad &\longrightarrow \quad K[X],\\
	(f,g) \quad \quad \quad \quad &\longmapsto \quad h(\nabla f , \nabla g),
\end{align*} is a symmetric biderivation structure on $K[X]$. 
\end{exam}

One can extend the usual notions of subalgebras, ideals, and quotient algebras to algebras endowed with a biderivation structure, as well as define morphisms between such algebras.

Let $(\mathscr{A},\cdot,\Psi)$ be an algebra with a biderivation $\Psi$. A \textbf{subalgebra} of $(\mathscr{A},\cdot,\Psi)$ is an associative subalgebra $\mathscr{B} \subseteq \mathscr{A}$ that is closed under $\Psi$; that is, $\Psi(x,y) \in \mathscr{B}$ for all $x,y \in \mathscr{B}$.

A subspace $I \subseteq \mathscr{A}$ is called an \textbf{ideal} of $(\mathscr{A},\cdot,\Psi)$ if it is an ideal of the associative algebra $(\mathscr{A},\cdot)$ and satisfies
\[
\Psi(x,y), \, \Psi(y,x) \in I, \quad \text{for all } x \in \mathscr{A}, \ y \in I.
\]
If $I$ is an ideal, the \textbf{quotient algebra} $\mathscr{A}/I$ inherits a biderivation structure $\overline{\Psi}$ defined by
\[
\overline{\Psi}(\overline{x},\overline{y}) := \overline{\Psi(x,y)}, \quad \text{for all} \ \overline{x},\overline{y} \in \mathscr{A}/I,
\]
which is well-defined, making $(\mathscr{A}/I, \cdot, \overline{\Psi})$ an algebra with the biderivation structure $\overline{\Psi}$.

Finally, $(\mathscr{A},\cdot,\circ)$ is called \textbf{simple} if it has no nontrivial ideals; that is, its only ideals are $0$ and $\mathscr{A}$ itself.

A \textbf{morphism}
\[
\varphi\colon (\mathscr{A},\cdot, \Psi_1) \longrightarrow (\mathscr{B},\cdot,\Psi_2)
\]
between algebras with biderivation structures is an associative algebra homomorphism satisfying
\[
\varphi\bigl(\Psi_1(x,y)\bigr) = \Psi_2\bigl(\varphi(x),\varphi(y)\bigr) \quad \text{for all } x,y \in \mathscr{A}.
\]
An \textbf{endomorphism} of $\mathscr{A}$ is a morphism from $\mathscr{A}$ to itself, and an \textbf{automorphism} of $\mathscr{A}$ is a bijective endomorphism whose inverse is also a morphism.

The group of automorphisms of $(\mathscr{A},\cdot,\Psi)$ is denoted by $\aut(\mathscr{A},\cdot,\Psi)$, or simply by $\aut(\mathscr{A})$.
\\

Now we introduce the polynomial algebra with the standard symmetric biderivation structure, which is the main object to study in this paper. The following notations will be kept henceforth.

1. $K$ is a field with $\ch (K) \neq 2$ thus contains $\frac{1}{2}$. 

2. $n\ge1$ is a positive integer. 


For $u=(a_1,\cdots,a_n)$ and $v=(b_1,\cdots,b_n)\in K^n$,  the inner product $u\cdot v=\sum_{i=1}^n a_ib_i$ is the \textbf{standard symmetric bilinear form} on $K^n$.

3. $\mathscr{A}[n]= K[x_1, \cdots, x_n]$. For all $f, g \in \mathscr{A}[n]$,
\begin{equation}
   f\circ g: = \sum_{i=1}^n \ptdx{f}{i}\ptdx{g}{i}= \nabla f \cdot \nabla g. 
\end{equation}\\[1mm]

We adopt the notation $\circ$, as the algebra $(\mathscr{A}[n],\circ)$ will later be shown to admit a close connection with a Jordan algebra.

It is straightforward to verify that the bilinear map
\[
\circ : \mathscr{A}[n] \times \mathscr{A}[n] \longrightarrow \mathscr{A}[n]
\]
defines a symmetric biderivation structure on $\mathscr{A}[n]$. Moreover, $1 \circ f = 0$ for all  $f \in \mathscr{A}[n]$.
We shall refer to $(\mathscr{A}[n],\cdot,\circ)$ as the polynomial algebra in $n$ variables with the \textbf{standard symmetric biderivation structure}.

In fact, this algebra arises as a special case of Example \ref{3}.
Here $X = \mathbb{A}^n(K)$ is the affine space over $K$, and $\mathscr{A}[n] = K[X]$ is its coordinate algebra. Let $h$ be the symmetric tensor field of type $(0,2)$ on $X$ such that, for each point $P \in \mathbb{A}^n(K)$, the restriction of $h$ to the tangent space $T_P(X) \cong K^n$ coincides with the standard symmetric bilinear form.

The product $\circ$ admits a natural geometric interpretation.
For $f \in \mathscr{A}[n]$, the gradient vector field with respect to the bilinear form $h$ is given by
\[
\nabla f = \sum_{i=1}^n \frac{\partial f}{\partial x_i}\,\frac{\partial}{\partial x_i}.
\]
Then, for any $f,g \in \mathscr{A}[n]$, one has
\begin{equation}\label{2}
f \circ g = \Bigg(\sum_{i=1}^n \frac{\partial f}{\partial x_i}\frac{\partial}{\partial x_i}\Bigg)(g) = (\nabla f)(g).
\end{equation}
By symmetry of the construction, the identity $f \circ g = (\nabla g)(f)$ also holds.

Let $ \mathscr{A}_i[n]$ be the subspace of homogeneous polynomials of degree $i$, then 
\[ \mathscr{A}[n] = \bigoplus_{i \geq 0}  \mathscr{A}_i[n];\ \quad \text{and} \quad  \mathscr{A}_i[n]  \cdot \mathscr{A}_j[n] \subset \mathscr{A}[n]_{i+j}, \ \forall i,j\ge0.\]
For any homogeneous $f, g$, we have $\deg (f \circ g) = \deg f + \deg g -2$,
hence 
\be
\label{7}\mathscr{A}_i[n] \circ \mathscr{A}_j[n] \subset \mathscr{A}_{i+j-2}[n],\ \forall i,j\ge0.
\ee

\begin{prop} \label{simp}
    Assume $\ch(K)=0$. Then $(\mathscr{A}[n],\cdot,\circ)$ is simple  for all $n\ge1$.
\end{prop}

\bp
Let $I \subseteq \mathscr{A}[n]$ be a nonzero ideal, and choose a nonzero element $f \in I$.
Let $ax_1^{i_1}\cdots x_n^{i_n}$ be one of the monomials of highest degree appearing in $f$ with coefficient $a \in K^\times$. Then, by successive applications of the product $\circ$, we obtain
\[
x_n^{\circ i_n} \circ \Big(x_{n-1}^{\circ i_{n-1}} \circ \big(\cdots \circ (x_1^{\circ i_1} \circ f)\cdots\big)\Big)  
= i_1! \cdots i_n! \, a \in I,
\]
where $x_k^{\circ i_k}$ denotes the $i_k$-fold $\circ$-power of $x_k$ ($k=1,\dots,n$).

Since $\ch(K) = 0$, the scalar $i_1!\cdots i_n!\, a$ is nonzero, hence $1 \in I$, and $I = \mathscr{A}[n]$.
\ep

It is clear that $(\mathscr{A}[i],\cdot,\circ)$ is a subalgebra of $(\mathscr{A}[i+1],\cdot,\circ)$ for all $i \geq 1$. We now explore some basic properties of $(\mathscr{A}[n],\circ)$.

The associator 
\[
[x_1^i,x_1^j,x_1^k] := (x_1^i \circ x_1^j)\circ x_1^k - x_1^i \circ (x_1^j \circ x_1^k) = ijk(i-k)\,x_1^{\,i+j+k-4},
\]
which is not always zero. Hence $(\mathscr{A}[n],\circ)$ is non-associative for all $n \geq 1$.

Moreover, since $x_1 \circ x_1 = 1 \neq 0$,
the operation $\circ$ does not satisfy the alternating property. On the other hand, the Jacobiator 
\begin{align*}
J(x_1^i,x_1^j,x_1^k) &:= (x_1^i \circ x_1^j)\circ x_1^k + (x_1^j \circ x_1^k)\circ x_1^i + (x_1^k \circ x_1^i)\circ x_1^j \\
&= 2ijk(i+j+k-3)\,x_1^{\,i+j+k-4},
\end{align*}
which is not always zero. Thus $(\mathscr{A}[n],\circ)$ is not a Lie algebra for all $n \geq 1$.

Then we would like to know if $(\mathscr{A}[n],\circ)$ is a Jordan algebra.

\section{The Jordan Algebra $\mathscr{A}_2[n]$ and the algebra $\mathscr{A}[n]$ as its bimodule} \label{SecJAJM}

Notice that $\mathscr{A}_2[n]$ is closed under the product $\circ$ by (\ref{7}). We will see that $(\mathscr{A}_2[n],\circ)$ is a Jordan algebra. Let us recall the definition of Jordan algebra.

\begin{defi}
A \textbf{Jordan algebra} is a vector space $J$ over a field $K$ together with a bilinear product $\circ: J \times J \to J$ such that
    \begin{align*}
    x \circ y &= y \circ x, \tag{J1} \\
    x \circ (y \circ x^2) &= (x \circ y) \circ x^2 \tag{J2} \label{J2}
    \end{align*}
hold for any $x, y \in J$.
\end{defi}

In $\mathscr{A}[n]$, it follows from (\ref{2}) that for all $f \in \mathscr{A}[n]$,
\[L_f = \nabla f:\mathscr{A}[n]\to \mathscr{A}[n].\]

For each $i \geq 0$, let
\[
\mathscr{A}_{\le i}[n] = \bigoplus_{k=0}^i \mathscr{A}_{k}[n]
\]
be the subspace of $\mathscr{A}[n]$ consisting of all polynomials of degree at most $i$.

For any polynomial $f \in K[x_1,\dots,x_n]$, we adopt the following shorthand for its partial derivatives:
\[
f_i := \frac{\partial f}{\partial x_i}, 
\qquad 
f_{ij} := \frac{\partial^2 f}{\partial x_i \partial x_j}, 
\qquad 
f_{ijk} := \frac{\partial^3 f}{\partial x_i \partial x_j \partial x_k}.
\]

\begin{thm}\label{ThmJA2}
Let $f,g \in \mathscr{A}[n]$. Then the commutator of their gradient vector fields is given by
\begin{equation}
    [\nabla f,\nabla g]
    =\sum_{k=1}^n \left( \sum_{i=1}^n f_i g_{ik} - g_i f_{ik} \right)\frac{\partial}{\partial x_k}.
\end{equation}
Moreover, for any $f \in \mathscr{A}[n]$, one has
\begin{equation}
    \big[\nabla f, \nabla (f \circ f)\big] \;=\; 2\sum_{i,j,k=1}^n f_i f_j f_{ijk}\,\frac{\partial}{\partial x_k}.
\end{equation}
In particular, if $f \in \mathscr{A}_{\le 2}[n]$, then
\[
\big[\nabla f, \nabla (f \circ f)\big] = 0.
\]
\end{thm}

\begin{proof}
Throughout the proof, all summation indices are understood to range over $\{1,\dots,n\}$.

The gradient vector fields of $f,g \in \mathscr{A}[n]$ are
\[
\nabla f = \sum_i f_i \,\frac{\partial}{\partial x_i}, 
\qquad 
\nabla g = \sum_j g_j \,\frac{\partial}{\partial x_j}.
\]
For two vector fields $A = \sum_i a^i \frac{\partial}{\partial x_i}$ and $B = \sum_k b^k \frac{\partial}{\partial x_k}$, their commutator is
\[
[A,B] = \sum_k \sum_i \Big(a^i \frac{\partial b^k}{\partial x_i} - b^i \frac{\partial a^k}{\partial x_i}\Big)\frac{\partial}{\partial x_k}.
\]
Applying this to the case when $a^i = f_i$, $b^i = g_i$. Noting that $\frac{\partial}{\partial x_i} g_k = g_{ik}$ and $\frac{\partial}{\partial x_i} f_k = f_{ik}$, we obtain
\[
[\nabla f, \nabla g] 
= \sum_{k}\Big(\sum_{i} f_i g_{ik} - g_i f_{ik}\Big)\frac{\partial}{\partial x_k},
\]
which is the first claim.

Next, set $g = f \circ f$. Then
\[
g = (\nabla f)(f) = \sum_j f_j^2.
\]
Differentiating gives
\[
g_i = \frac{\partial g}{\partial x_i} = 2\sum_j f_j f_{ji}, 
\qquad
g_{ik} = \frac{\partial g_i}{\partial x_k} = 2\sum_j f_{jk} f_{ji} + 2\sum_j f_j f_{jik}.
\]
Substituting these into the formula for $[\nabla f,\nabla g]$ yields
\[
[\nabla f,\nabla g] 
= 2\sum_k \left(\sum_{i,j} \big(f_i f_{ij} f_{jk} + f_i f_j f_{ijk}\big) - \sum_{i,j} f_j f_{ji} f_{ki}\right)\frac{\partial}{\partial x_k}.
\]
Since
\[
\sum_{i,j} f_i f_{ij} f_{jk} = \sum_{i,j} f_j f_{ji} f_{ki}
\]
by renaming indices, the first and third sums cancel. Hence,
\[
[\nabla f, \nabla(f \circ f)] = 2\sum_{i,j,k} f_i f_j f_{ijk}\,\frac{\partial}{\partial x_k}.
\]
Finally, if $f \in \mathscr{A}_{\le 2}[n]$, then all third-order derivatives vanish, i.e., $f_{ijk} = 0$. Therefore,
\[
[\nabla f, \nabla(f \circ f)] = 0,
\]
as required.
\end{proof}

Let $H_n(K)$ denote the special Jordan algebra of $n\times n$ symmetric matrices over $K$ with product
\[
A \circ B := \tfrac{1}{2}(AB+BA), \qquad A,B \in H_n(K).
\]
For $A \in H_n(K)$, let $q_A = XAX^T \in \mathscr{A}_2[n]$, where $X=(x_1,\dots,x_n)$.

\begin{prop}
The algebra $(\mathscr{A}_{\le 2}[n], \circ)$ is a Jordan algebra, and $(\mathscr{A}_2[n], \circ)$ is a Jordan subalgebra. Moreover:
\begin{enumerate}
    \item[1.] $\mathscr{A}_0[n]$ is a Jordan ideal of $(\mathscr{A}_{\le 2}[n], \circ)$;
    \item[2.] $\mathscr{A}_{\le 1}[n]$ is the radical of $(\mathscr{A}_{\le 2}[n], \circ)$;
    \item[3.] There is an isomorphism $\mathscr{A}_{\le 2}[n]/\mathscr{A}_{\le 1}[n] \cong \mathscr{A}_2[n]$ of Jordan algebras;
    \item[4.] The map
    \[
    \xi : \mathscr{A}_2[n] \longrightarrow H_n(K), 
   \qquad q_A \mapsto 4A,
    \]
    is an isomorphism of Jordan algebras.
\end{enumerate}
In particular, $\mathscr{A}_2[n]$ is a simple Jordan algebra with unit
\[
e = \tfrac{1}{4}(x_1^2 + \cdots + x_n^2).
\]
\end{prop}

\bp

The first statement follows from Theorem \ref{ThmJA2} and from the fact that $\mathscr{A}_2[n]$ is closed under $\circ$.
For any $f \in \mathscr{A}_0[n]$ and $g \in \mathscr{A}_{\leq 2}[n]$, since $\deg(f \circ g) = 0$, the subalgebra $\mathscr{A}_0[n]$ is a Jordan ideal of $\mathscr{A}_{\leq 2}[n]$.

Next We show that $\xi$ is an isomorphism of Jordan algebras. 

Firstly, it is a linear bijection.
Next, we show that $\xi$ is a homomorphism of Jordan algebras. Let $A = (a_{ij})$ and $B = (b_{ij})$ be symmetric matrices in $H_n(K)$, and set
\[
C := A \circ B = \tfrac{1}{2}(AB + BA) = (c_{ij}).
\]
Then $C$ is symmetric, with entries
\[
c_{ij} = \frac{1}{2}\left(\sum_{k=1}^n a_{ik}b_{kj} + \sum_{k=1}^n a_{jk}b_{ki}\right),
\]
and in particular $c_{ii} = \sum_{k=1}^n a_{ik}b_{ki}$.

On the other hand, for the corresponding quadratic forms $q_A, q_B \in \mathscr{A}_2[n]$, one computes
\begin{align*}
	q_A \circ q_B &= \sum_i \ptdx{q_A}{i}\ptdx{q_B}{i} = \sum_i\left(\sum_j 2a_{ij}x_j\right) \left(\sum_k 2b_{ik}x_k\right)\\
	&= \sum_{j,k} \left(\sum_i 4a_{ij}b_{ik}\right)x_jx_k\\
	&= \sum_j \left(\sum_i4a_{ij}b_{ij}\right)x_j^2 + 2 \sum_{j<k} \left(\sum_i 2a_{ij}b_{ik} + \sum_i 2a_{ik}b_{ij} \right)x_jx_k\\
	&= 4\sum_{j,k} c_{jk} x_jx_k = 4q_C.
\end{align*}
Hence, by definition of $\xi$, we have $\xi(q_A \circ q_B) = 16 C$, and
\[
\xi(q_A) \circ \xi(q_B) = 4A \circ 4B = 16(A \circ B) = 16 C = \xi(q_A \circ q_B).
\]
Therefore, the map $\xi$ is a linear bijection that preserves the Jordan product, thus is an isomorphism of Jordan algebras.

Since $H_n(K)$ is a simple Jordan algebra with unit $I_n$ (the identity matrix), it follows that $\mathscr{A}_2[n]$ is also simple, with unit
\[
\xi^{-1}(I_n) = e = \frac{1}{4}(x_1^2 + \cdots + x_n^2).
\]

Next, we show that $\mathscr{A}_{\leq 1}[n] = \mathrm{rad}(\mathscr{A}_{\leq 2}[n])$.
By \eqref{7}, it is an ideal of $\mathscr{A}_{\leq 2}[n]$, and
\[
(\mathscr{A}_{\leq 1}[n])^{\circ 3} = 0,
\]
so $\mathscr{A}_{\leq 1}[n]$ is a nilpotent ideal.

On the other hand, there are Jordan algebra isomorphisms
\[
\mathscr{A}_{\leq 2}[n] / \mathscr{A}_{\leq 1}[n] \simeq \mathscr{A}_2[n] \simeq H_n(K),
\]
and since $H_n(K)$ is simple, $\mathscr{A}_{\leq 1}[n]$ is the maximal nilpotent ideal. Hence $\mathscr{A}_{\leq 1}[n]$ coincides with the radical of $\mathscr{A}_{\leq 2}[n]$.
\ep

Note that $(\mathscr{A}[n],\circ)$ is not a Jordan algebra.

The algebra $\mathscr{A}[n]$ carries a natural $\mathscr{A}_2[n]$-bimodule structure via
\begin{align*}
	\mathscr{A}_2[n] \otimes \mathscr{A}[n] \to \mathscr{A}[n], &\quad f \otimes h \mapsto f \circ h;\\
	\mathscr{A}[n] \otimes \mathscr{A}_2[n] \to \mathscr{A}[n], &\quad h \otimes f \mapsto h \circ f.
\end{align*}
Moreover, each homogeneous component $\mathscr{A}_k[n]$ is an $\mathscr{A}_2[n]$-subbimodule.

We would like to know if  $\mathscr A_k[n]$ and  $\mathscr A[n]$ are Jordan bimodules. Let us recall the notion of Jordan bimodules (cf. \cite{carotenuto2019,jacobson1954}):

\begin{defi}
    A \textbf{Jordan bimodule} over a Jordan algebra $J$ is a vector space with two bilinear maps
    \begin{align*}
        J \otimes M \to M, &\quad x \otimes m \mapsto xm:=L_x(m);\\
        M \otimes J \to M, &\quad m \otimes x \mapsto mx:=R_x(m).
    \end{align*}
    such that the direct sum $J \oplus M$ with the product
    \[(x, m)\circ(x^\prime, m^\prime) := (x\circ x^\prime, xm^\prime+mx^\prime)\]
    is a Jordan algebra.
\end{defi}

Denote $x^2:=x\circ x$. This definition is equivalent to requiring the following identities hold :
\begin{equation} \label{JM}
    \begin{aligned}
        xm &= mx,\\
        x(x^2 m) &= x^2 (xm),\\
        x^2 (ym) + 2(x\circ y)(xm) &= (x^2\circ y)m + 2(x(y(xm)))
    \end{aligned}
\end{equation}
for any $x,y \in J$, $m \in M$ (see \cite{carotenuto2019}).

It is well known that if $I$ is a Jordan ideal of a Jordan algebra $J$, then $I$ naturally becomes a Jordan $J$-bimodule via the Jordan product $\circ$. In particular, any Jordan algebra $J$ is a Jordan bimodule over itself.


\begin{thm} \label{ThmJM}
   Among the $\mathscr{A}_2[n]$-bimodules $\mathscr{A}_i[n]$,  only $\mathscr{A}_0[n] = K$, $\mathscr{A}_1[n]$, and $\mathscr{A}_2[n]$ are Jordan bimodules. 
\end{thm}

\begin{proof}
It is clear that $\mathscr{A}_0[n]$ is a Jordan bimodule over $\mathscr{A}_2[n]$, and that $\mathscr{A}_2[n]$ is a Jordan bimodule over itself.
Thus, it remains to show that: i) $\mathscr{A}_i[n]$ is not a Jordan bimodule over $\mathscr{A}_2[n]$ for $i \ge 3$;
ii) $\mathscr{A}_1[n]$ is a Jordan bimodule over $\mathscr{A}_2[n]$.

i) It suffices to find specific elements $f, g \in \mathscr{A}_2[n]$ and $h \in \mathscr{A}_i[n]$ for which the Jordan bimodule identity \eqref{JM} fails.

Let $f = g = x_1^2 \in \mathscr{A}_2[n]$ and $h = x_1^i$. Consider
\[
S := (f^{\circ 2} \circ g) \circ h - f^{\circ 2} \circ (g \circ h) - 2(f \circ g) \circ (f \circ h) + 2(f \circ (g \circ (f \circ h))).
\]
A direct calculation yields
\[
S = 16\,i(i-1)(i-2)\, x_1^i,
\]
which vanishes if and only if $i = 0, 1, 2$.

Hence, Jordan bimodules can only occur among $\mathscr{A}_0[n] = K$, $\mathscr{A}_1[n]$, and $\mathscr{A}_2[n]$.
Since $\mathscr{A}_0[n]$ and $\mathscr{A}_2[n]$ are already known to be Jordan bimodules over $\mathscr{A}_2[n]$, it remains to check $\mathscr{A}_1[n]$.

ii) Since $\mathscr{A}_0[n]$ is a Jordan ideal of $\mathscr{A}_{\le 2}[n]$, the quotient $\mathscr{A}_{\le 2}[n] / \mathscr{A}_0[n]$ inherits a natural Jordan algebra structure, and we have the short exact sequence
\[
0 \longrightarrow \mathscr{A}_0[n] \longrightarrow \mathscr{A}_{\le 2}[n] \longrightarrow \mathscr{A}_{\le 2}[n] / \mathscr{A}_0[n] \longrightarrow 0.
\]
This yields a linear isomorphism
\[
\eta \colon \mathscr{A}_2[n] \oplus \mathscr{A}_1[n] \;\overset{\sim}{\longrightarrow}\; \mathscr{A}_{\le 2}[n] / \mathscr{A}_0[n], \quad (a_2, a_1) \mapsto \overline{a_2 + a_1}.
\]

Through this isomorphism, we may define a Jordan algebra structure on $\mathscr{A}_2[n] \oplus \mathscr{A}_1[n]$ by
\[
(a_2, a_1) \cdot (b_2, b_1) := \bigl(a_2 \circ b_2, \; a_1 \circ b_2 + a_2 \circ b_1\bigr).
\]
Indeed, for any $(a_2, a_1), (b_2, b_1) \in \mathscr{A}_2[n] \oplus \mathscr{A}_1[n]$, one has
\begin{align*}
\eta\bigl((a_2, a_1)\bigr) \circ \eta\bigl((b_2, b_1)\bigr)
&= \overline{a_2 + a_1} \circ \overline{b_2 + b_1} \\
&= \overline{a_2 \circ b_2 + \bigl(a_1 \circ b_2 + a_2 \circ b_1 + a_1 \circ b_1\bigr)} \\
&= \overline{a_2 \circ b_2 + \bigl(a_1 \circ b_2 + a_2 \circ b_1\bigr)} \\
&= \eta\bigl((a_2, a_1) \cdot (b_2, b_1)\bigr).
\end{align*}
Hence the defined product endows $\mathscr{A}_2[n] \oplus \mathscr{A}_1[n]$ with the structure of a Jordan algebra. In particular, $\mathscr{A}_1[n]$ becomes a Jordan bimodule over $\mathscr{A}_2[n]$.
\end{proof}

It follows that the $\mathscr{A}_2[n]$-bimodule $\mathscr{A}[n]$ is not a Jordan bimodule, since its subbimodules $\mathscr{A}_k[n]$ fail to be Jordan bimodules for $k \ge 3$.

\section{Weight Space Decomposition of $\mathscr{A}[n]$ } 
\label{SecRSD}

We shall show that $\mathscr{A}[n]$ as a bimodule over $\mathscr{A}_2[n]$ admits a weight space decomposition analogous to that in the representation theory of semisimple Lie algebras.

Recall the isomorphism of Jordan algebras
\[
\xi \colon \mathscr{A}_2[n] \to H_n(K), \quad q_A \mapsto 4A.
\]
Let $\{e_{ij} \mid i,j = 1,\dots,n\}$ denote the standard matrix units in $M_n(K)$. It is well known that the set $\{e_{11}, \dots, e_{nn}\}$ forms a complete set of mutually orthogonal primitive idempotents in $H_n(K)$. (An idempotent is called primitive if it cannot be expressed as the sum of two nonzero orthogonal idempotents.) The $K$-linear span of $\{e_{11}, \dots, e_{nn}\}$ is the space of diagonal matrices in $H_n(K)$, which is a Cartan subalgebra of $H_n(K)$; for the definition of a Cartan subalgebra, see \cite{jacobson1968}.

Since $\xi^{-1}(e_{ii}) = \frac{1}{4}x_i^2$, the set
\[
\left\{\frac{1}{4}x_1^2, \dots, \frac{1}{4}x_n^2\right\}  
\]
forms a complete set of mutually orthogonal primitive idempotents in $\mathscr{A}_2[n]$.

Let $\mathscr{H}$ denote the $K$-linear span of $\frac{1}{4}x_i^2$ for $i = 1, \dots, n$. Define $\beta_1, \dots, \beta_n \in \mathscr{H}^*$ to be the basis of the dual space $\mathscr{H}^*$ corresponding to the basis $\frac{1}{4}x_1^2, \dots, \frac{1}{4}x_n^2$ of $\mathscr{H}$; that is,
\[
\beta_i\Big(\frac{1}{4}x_j^2\Big) = \delta_{ij}.
\]

For a monomial $X^u := x_1^{u_1} \cdots x_n^{u_n}$ with $u = (u_1, \dots, u_n)$, one has
\[
x_i^2 \circ X^u = 2 u_i X^u \quad \text{and} \quad \frac{1}{4} x_i^2 \circ X^u = \frac{1}{2} u_i X^u.
\]

If we define the linear maps 
\begin{align*}
    \pi\colon \mathscr A_2[n]\to \mathrm{End}(\mathscr A[n]), &\quad f\mapsto (h \mapsto f \circ h);\\
    \pi_k\colon \mathscr A_2[n]\to \mathrm{End}(\mathscr A_k[n]), &\quad f\mapsto (h \mapsto f \circ h),
\end{align*}
then for any $h = \sum_{i=1}^n a_i \frac{x_i^2}{4} \in \mathscr{H}$, we have
\begin{equation} \label{EigOfH}
    h \circ X^u = \left( \sum_{i=1}^n \frac{1}{2} a_i u_i \right) X^u = \left( \sum_{i=1}^n \frac{1}{2} u_i \beta_i \right)(h) X^u,
\end{equation}
showing that $X^u$ is an eigenvector of $\pi(h)$ with eigenvalue $\big(\sum_i \frac{1}{2} u_i \beta_i\big)(h)$.

For $\alpha \in \mathscr{H}^*$, define
\[
\mathscr{A}[n]_{\alpha} := \left\{ f \in \mathscr{A}[n] \mid h \circ f = \alpha(h) f \text{ for all } h \in \mathscr{H} \right\}.
\]
If $\mathscr{A}[n]_{\alpha} \neq 0$, then $\alpha$ is called a \textbf{weight} of $\mathscr{A}[n]$ with respect to $\mathscr{H}$, and $\mathscr{A}[n]_{\alpha}$ is called the $\alpha$-\textbf{weight space}.

From \eqref{EigOfH}, for a monomial $X^u = x_1^{u_1} \cdots x_n^{u_n}$, we have
\[
X^u \in \mathscr{A}[n]_{\beta}, \quad \beta = \frac{1}{2} \sum_{i=1}^n u_i \beta_i.
\]
Since $\mathscr{A}[n] = \bigoplus_{u \in (\mathbb{Z}_{\ge 0})^n} K \cdot X^u$, it follows that for $\beta = \frac{1}{2} \sum_{i=1}^n u_i \beta_i$ with $u_i \ge 0$,
\[
\mathscr{A}[n]_{\beta} = K \cdot X^u = K \cdot x_1^{u_1} \cdots x_n^{u_n},
\]
which is 1-dimensional. Hence the set of weights of $\mathscr{A}[n]$ with respect to $\mathscr{H}$ is
\[
W = \operatorname{Span}_{\mathbb{Z}_{\ge 0}} \Big\{ \tfrac{1}{2} \beta_1, \dots, \tfrac{1}{2} \beta_n \Big\},
\]
a free commutative monoid generated by $\left\{\frac{1}{2}\beta_1, \dots, \frac{1}{2}\beta_n\right\}$, and
\begin{equation} \label{WSP1}
    \mathscr{A}[n] = \bigoplus_{\beta \in W} \mathscr{A}[n]_{\beta}
\end{equation}
is a direct sum of 1-dimensional weight spaces.

One has
\[
\mathscr{A}[n]_{\alpha_1} \cdot \mathscr{A}[n]_{\alpha_2} = \mathscr{A}[n]_{\alpha_1 + \alpha_2}, \quad \forall \alpha_1, \alpha_2 \in W,
\]
since $\circ$ is a biderivation on $\mathscr{A}[n]$. Moreover, it is direct to verify that
\[
\mathscr{A}[n]_{\alpha_1} \circ \mathscr{A}[n]_{\alpha_2} \subseteq \bigoplus_{i=1}^n \mathscr{A}[n]_{\alpha_1 + \alpha_2 - \beta_i}, \quad \forall \alpha_1, \alpha_2 \in W.
\]
For $k \ge 0$, set
\[
W_k := \Big\{ \sum_{i=1}^n \frac{1}{2} u_i \beta_i \ \Big| \ \sum_{i=1}^n u_i = k, u_i\in \mathbb Z_{\ge0} \Big\},
\]
then
\[
\mathscr{A}_k[n] = \bigoplus_{\beta \in W_k} \mathscr{A}[n]_{\beta}.
\]
In particular,
\[
W_2 = \Big\{ \tfrac{1}{2}\beta_i + \tfrac{1}{2}\beta_j \ \Big| \ 1 \le i \le j \le n \Big\}, \quad \text{and} \quad \mathscr{A}_2[n] = \bigoplus_{\beta \in W_2} \mathscr{A}[n]_{\beta},
\]
which coincides with the Peirce decomposition of $\mathscr{A}_2[n]$ relative to the set of orthogonal idempotents $\left\{ \frac{1}{4}x_1^2, \dots, \frac{1}{4}x_n^2 \right\}$.

To summarize:

\begin{thm} \label{ThmWSD}
    With respect to the action of $\mathscr{H}$, the algebra $\mathscr{A}[n]$ and its homogeneous subspaces $\mathscr{A}_k[n]$ admit the weight space decomposition
    \[
    \mathscr{A}[n] = \bigoplus_{\beta \in W} \mathscr{A}[n]_{\beta}, \qquad 
    \mathscr{A}_k[n] = \bigoplus_{\beta \in W_k} \mathscr{A}[n]_{\beta},
    \]
    where
    \[
    W = \operatorname{Span}_{\mathbb{Z}_{\ge 0}} \Big\{ \tfrac{1}{2}\beta_1, \dots, \tfrac{1}{2}\beta_n \Big\} \subset \mathscr{H}^*, 
    \qquad 
    W_k = \left\{ \sum_{i=1}^n \tfrac{1}{2} u_i \beta_i \ \Big| \ \sum_i u_i = k,  u_i\in \mathbb Z_{\ge0}  \right\},
    \]
    and for $\beta = \sum_{i=1}^n \frac{1}{2} u_i \beta_i$, one has $\mathscr{A}[n]_{\beta} = K x_1^{u_1} \cdots x_n^{u_n}$.

    Moreover, for any $\alpha_1, \alpha_2 \in W$,
    \[
    \mathscr{A}[n]_{\alpha_1} \cdot \mathscr{A}[n]_{\alpha_2} = \mathscr{A}[n]_{\alpha_1 + \alpha_2}, 
    \]
    and
    \[
    \mathscr{A}[n]_{\alpha_1} \circ \mathscr{A}[n]_{\alpha_2} \subseteq \bigoplus_{i=1}^n \mathscr{A}[n]_{\alpha_1 + \alpha_2 - \beta_i}.
    \]
\end{thm}

The element $e = \frac{1}{4}(x_1^2 + \cdots + x_n^2) \in \mathscr{A}_2[n]$ acts as
\[
\pi(e) = \frac{1}{2} \sum_{i=1}^n x_i \frac{\partial}{\partial x_i},
\]
which is half the Euler operator. Its restriction to $\mathscr{A}_k[n]$ is
\[
\pi_k(e) = \frac{k}{2} \, \mathrm{id}_{\mathscr{A}_k[n]}.
\]

For distinct indices $i \neq j$, one has
\[
\pi(x_i x_j) = \nabla(x_i x_j) = x_i \frac{\partial}{\partial x_j} + x_j \frac{\partial}{\partial x_i}.
\]

\begin{thm} \label{simpMod}
    Assume $\ch(K) = 0$. Then for all $n \ge 1$ and $k \ge 0$, the $\mathscr{A}_2[n]$-bimodule $\mathscr{A}_k[n]$ is simple.
\end{thm}

\bp
The cases $k=0$ or $n=1$ are clear. Assume $n \ge 2$ and $k \ge 1$.

Let
\[
T_k = \left\{(u_i)_{i=1}^n \in (\mathbb{Z}_{\ge 0})^n \mid \sum_{i=1}^n u_i = k\right\},
\]
and for $u, v \in T_k$ with $u = (u_i)_{i=1}^n$ and $v = (v_i)_{i=1}^n$, define
\[
d(u, v) := \sum_{i=1}^n |u_i - v_i|.
\]
Every nonzero element in $\mathscr{A}_k[n]$ can be written as
\[
f = \sum_{u \in S} \lambda_u X^u, \quad 0 \ne \lambda_u \in K, \ \varnothing \ne S \subseteq T_k.
\]

Let $M \subseteq \mathscr{A}_k[n]$ be a nonzero $\mathscr{A}_2[n]$-subbimodule, and take $0 \ne f = \sum_{u \in S} \lambda_u X^u \in M$.
For any $h \in \mathscr{H}$,
\[
\pi(h)(X^u) = \frac{1}{2} \sum_{i=1}^n u_i \beta_i(h) \cdot X^u.
\]
Since $\ch(K) = 0$, we can choose $y_1, \dots, y_n \in K$ such that the values
\[
\left\{ \frac{1}{2} \sum_{i=1}^n u_i y_i \ \big| \ u \in S \right\}
\]
are pairwise distinct. Let $h_0 \in \mathscr{H}$ satisfy $\beta_i(h_0) = y_i$. Since $M$ is $\pi(h_0)$-invariant, it follows that $X^u \in M$ for every $u \in S$.

Next, we show that $X^u \in M$ for all $u \in T_k$; hence $M = \mathscr{A}_k[n]$, proving that $\mathscr{A}_k[n]$ is simple.
Assume, for contradiction, that there exists $u \in T_k$ such that $X^u \notin M$. Choose $X^v \in M$ satisfying
\[
d(u, v) = \min \{ d(u, w) \mid X^w \in M \}.
\]
Then $d(u, v) > 0$. As $u, v \in T_k$, there exist indices $i$ and $j$ such that $u_i < v_i$ and $u_j > v_j$. Consider
\[
\pi(x_i x_j)(X^v) = \left( x_j \frac{\partial}{\partial x_i} + x_i \frac{\partial}{\partial x_j} \right)(X^v) = v_i X^{v'} + v_j X^{v''},
\]
where
\[
X^{v'} = x_1^{v_1} \cdots x_i^{v_i-1} \cdots x_j^{v_j+1} \cdots x_n^{v_n}, \quad 
X^{v''} = x_1^{v_1} \cdots x_i^{v_i+1} \cdots x_j^{v_j-1} \cdots x_n^{v_n}.
\]
As $\ch(K) = 0$, we have $v_i \ne 0$. By the previous argument, $X^{v'} \in M$. But
\[
d(u, v') = d(u, v) - 2 < d(u, v),
\]
which contradicts the minimality of $d(u, v)$. Therefore, no such $u$ exists, and $M = \mathscr{A}_k[n]$.
\ep

\section{The subgroup $\aut^0(\mathscr{A}[n],\cdot,\circ )$ of $\aut(\mathscr{A}[n],\cdot,\circ )$} 
\label{SecAutGp}

We denote  $(\mathscr{A}[n],\cdot,\circ)$ simply by $\mathscr{A}[n]$. Recall that $\aut(\mathscr{A}[n],\cdot,\circ)$ is the group of automorphisms of $(\mathscr{A}[n],\cdot,\circ)$, which  will be abbreviated as $\aut(\mathscr{A}[n])$.

Let $\varphi \in \aut(\mathscr{A}[n])$. Since $\varphi$ is an algebra endomorphism of $(\mathscr{A}[n],\cdot)$, it is completely determined by its values on the generators $x_1, \dots, x_n$. Set $\varphi(x_i) = h_i$. Then, for any $f \in \mathscr{A}[n]$,
\[
\varphi(f) = f(h_1, \dots, h_n),
\]
and, moreover, the images satisfy
\[
h_i \circ h_j = \varphi(x_i \circ x_j) = \delta_{ij}.
\]

The group $\aut(\mathscr{A}[n])$ contains the subgroup
\[
\aut^0(\mathscr{A}[n]) := \{\varphi \in \aut(\mathscr{A}[n]) \mid \varphi(\mathscr{A}_i[n]) \subseteq \mathscr{A}_i[n] \text{ for all } i\},
\]
the automorphisms that preserve each graded component $\mathscr{A}_i[n]$. Its structure can be explicitly described using the orthogonal group
\[
O(n,K) = \{A \in GL(n,K) \mid A^T A = I_n\}.
\]

\begin{thm} \label{ThmAut}
There is a group isomorphism
\[
\aut^0(\mathscr{A}[n]) \cong O(n, K).
\]
\end{thm}

\bp
By definition, for any $\varphi \in \operatorname{Aut}^0(\mathscr{A}[n])$, we have
\begin{equation}
    \varphi(x_j) = \sum_{k=1}^n a_{kj} x_k, \quad j = 1, \dots, n. \label{e14}
\end{equation}
Setting $h_j = \sum_{k=1}^n a_{kj} x_k$, the condition $\varphi \in \operatorname{Aut}(\mathscr{A}[n])$ requires
\[
\delta_{ij} = h_i \circ h_j = \sum_{k=1}^n a_{ki} a_{kj}.
\]
This shows that the matrix $A = (a_{ij})$ satisfies
$A^T A = I_n$, so that $A \in O(n,K)$.

Define a map 
    \begin{align*}
        \Phi\colon \aut^0(\mathscr{A}[n]) \longrightarrow O(n,K),\quad 
        \varphi \mapsto  A= (a_{ij})
    \end{align*} 
such that (\ref{e14}) holds.
We claim that $\Phi$ is a group homomorphism.

Let $\Phi(\varphi) = A = (a_{ij})$ and $\Phi(\psi) = B = (b_{ij})$, with $A, B \in O(n,K)$. Then
\begin{align*}
	(\psi \varphi)(x_i) &= \psi(h_i) = \psi\left(\sum_{k=1}^n a_{ki} x_k\right)\\
	&= \sum_{k=1}^n a_{ki} \psi(x_k) = \sum_{j=1}^n \left( \sum_{k=1}^n b_{jk} a_{ki} \right) x_j,
\end{align*}
so that
\[
\Phi(\psi \circ \varphi) = BA = \Phi(\psi) \Phi(\varphi).
\]
Moreover, $\Phi$ is injective, since its kernel consists only of the identity automorphism of $\mathscr{A}[n]$. It remains to show that $\Phi$ is surjective.

Given $A = (a_{ij}) \in O(n,K)$, define $\varphi: \mathscr{A}[n] \to \mathscr{A}[n]$ by
\[
\varphi(x_j) = \sum_{i=1}^n a_{ij} x_i := h_j, \quad j = 1, \dots, n,
\]
and extend $\varphi$ to all $f \in \mathscr{A}[n]$ via
\[
\varphi(f) = f(h_1, \dots, h_n).
\]
Then $\varphi$ is an automorphism of the polynomial algebra $(\mathscr{A}[n], \cdot)$. It remains to verify that $\varphi$ preserves the biderivation structure. Let $H = (h_1, \dots, h_n)$. For any $f,g \in \mathscr{A}[n]$,
\begin{align*}
		\varphi(f) \circ \varphi(g) &= f(h_1, \dots, h_n) \circ g(h_1, \dots, h_n)\\ 
		&= \sum_{j=1}^n \left( \sum_{i=1}^n \frac{\partial f}{\partial x_i}(H) \frac{\partial h_i}{\partial x_j} \right) \left( \sum_{k=1}^n \frac{\partial g}{\partial x_k}(H) \frac{\partial h_k}{\partial x_j} \right)\\
		&= \sum_{i,k=1}^n \left(\frac{\partial f}{\partial x_i}(H) \frac{\partial g}{\partial x_k}(H) \sum_{j=1}^n \frac{\partial h_i}{\partial x_j} \frac{\partial h_k}{\partial x_j}\right)= \sum_{i,k=1}^n \left(\frac{\partial f}{\partial x_i}(H) \frac{\partial g}{\partial x_k}(H) \delta_{ik}  \right)   \\
		&= \sum_{i=1}^n \frac{\partial f}{\partial x_i}(H) \frac{\partial g}{\partial x_i}(H)= \varphi(f\circ g).
	\end{align*}
Hence $\varphi$ is an automorphism of $\mathscr{A}[n]$, i.e., $\Phi(\varphi) = A$, which completes the proof.
\ep

Finally, we determine $\aut(\mathscr{A}[n])$ for $n=1$.

Let $\varphi \in \aut(\mathscr{A}[1])$ and write $\varphi(x) = h$. Then
\[
\varphi(x^{\circ 2}) = h^{\circ 2} = h' \cdot h' = 1.
\]
It follows that $h(x) = \lambda x + \mu$, where $\lambda = \pm 1$ and $\mu \in K$.
Therefore,
\[
\aut(\mathscr{A}[1]) \cong (K, +) \rtimes_\theta \{\pm 1\},
\]
where $\theta(\lambda): K \to K$ acts by $\mu \mapsto \lambda \mu$.

For $n > 1$, determining $\aut(\mathscr{A}[n])$ appears to be significantly more difficult.

\end{document}